\documentclass[12pt]{amsart}
\usepackage{amsmath,amssymb,amsfonts,amsthm,amscd,indentfirst}
\usepackage{indentfirst}
\usepackage{hyperref}
\usepackage{graphicx}
\usepackage{color}
\usepackage[dvipsnames]{xcolor}
\usepackage{ stmaryrd }
\numberwithin{equation}{section}

\usepackage{setspace}
\usepackage{amsmath, amssymb,amsthm, amsfonts, color}

\usepackage{amsmath}
\usepackage{amsthm}
\usepackage{amsfonts}
\usepackage{amssymb}
\usepackage{graphicx}
\usepackage{comment}
\usepackage{mathrsfs}
\usepackage{dsfont}
\usepackage{hyperref}
\usepackage[shortlabels]{enumitem}
\hypersetup{
    colorlinks=true,
    linkcolor=blue,
    filecolor=magenta,      
    urlcolor=cyan,
}

\usepackage{pgfplots}
\usepgfplotslibrary{polar}
\usepgflibrary{shapes.geometric}
\usetikzlibrary{calc}
\pgfplotsset{my style/.append style={axis x line=middle, axis y line=middle, xlabel={$x$}, ylabel={$y$}, axis equal }} %
\usepackage{tikz}
\usetikzlibrary{calc,arrows,hobby}

\newtheorem{thm}{Theorem}[section]
\newtheorem{lem}{Lemma}[section]
\newtheorem{prop}[thm]{Proposition}
\newtheorem{cor}[thm]{Corollary}

\theoremstyle{remark}
\newtheorem{rmk}{Remark}[section]

\renewcommand\d{\partial}

\newcommand{\lsm}{\llbracket}
\newcommand{\rsm}{\rrbracket}

\newcommand\beq{\begin{equation}}
\newcommand\eeq{\end{equation}}
\newcommand\ben{\begin{enumerate}}
\newcommand\een{\end{enumerate}}
\newcommand\bit{\begin{itemize}}
\newcommand\eit{\end{itemize}}

\DeclareMathOperator{\Hess}{Hess}
\DeclareMathOperator{\dv}{div}

\DeclareMathOperator{\tr}{tr}

\newcommand{\R}{\mathbb{R}}

\newcommand{\F}{\mathcal{F}_+(\Sigma,\gamma)}

\newcommand{\FF}{\mathcal{F}(\Sigma,\gamma)}

\def\mb{\mathfrak{m}_{_B}}
\def\p{\partial}

\def\m{\mathfrak{m}}
\newcommand{\ts}{b}
\newcommand{\Hs}{\mathcal{H}^2_s}
\newcommand{\Ho}{\mathcal{H}^2_1}

\newcommand\norm[1]{\left\lVert#1\right\rVert}

\newcounter{mnotecount}

\def\be{\begin{equation}}
\def\ee{\end{equation}}

\begin{document}

\date{}

\title{Estimates of the Bartnik mass}

\author{Pengzi Miao}
\address[Pengzi Miao]{Department of Mathematics, University of Miami, Coral Gables, FL 33146, USA}
\email{pengzim@math.miami.edu}

\author{Annachiara Piubello}
\address[Annachiara Piubello]{Department of Mathematics, University of Miami, Coral Gables, FL 33146, USA}
\email{a.piubello@math.miami.edu}

\thanks{P. Miao's research was partially supported by NSF grant DMS-1906423.}

\begin{abstract}
Given a metric $\gamma$ of nonnegative Gauss curvature and a positive function $H$ on a $2$-sphere $\Sigma$, we estimate the Bartnik quasi-local mass of $(\Sigma, \gamma, H)$ in terms of the 
area, the total mean curvature, and a quantity depending only on $\gamma$, measuring 
the roundness of the metric. If $\gamma$ has positive Gauss curvature, 
the roundness of $\gamma$ in the estimate is controlled by the ratio $\kappa$ between the maximum and the minimum of the Gauss curvature. 
As $\kappa \to 1$, the estimate approaches a sharp estimate for round spheres with arbitrary, 
positive mean curvature functions. 

Enroute we observe an estimate of the supremum of the total mean curvature among nonnegative 
scalar curvature fill-ins of a closed manifold with positive scalar curvature.
\end{abstract}

\maketitle

\markboth{Pengzi Miao and Annachiara Piubello}{Estimates of the Bartnik mass}

\section{introduction}

Given a two-sphere $\Sigma$,  a Riemannian metric $\gamma $ and a function $H $ on $\Sigma$, the Bartnik quasi-local mass 
\cite{B2, B3} of the triple $(\Sigma, \gamma, H)$ is given by 
\be
\mb(\Sigma , \gamma ,H) = \inf \left \lbrace \m (M, g) \,\vert \,(M, g) \text{ is an admissible extension of }(\Sigma, \gamma,H)\right \rbrace.
\ee
Here $ \m (\cdot)$ denotes the ADM mass functional \cite{ADM} and $(M, g)$, 
an asymptotically flat $3$-manifold  with boundary $\p M$, is an {admissible extension of $(\Sigma, \gamma,H)$} if
\begin{itemize}
\item[(i)]  $g$  has nonnegative scalar curvature;

\vspace{.1cm}

\item[(ii)]   $\p M$ with the induced metric is isometric to  $(\Sigma, \gamma)$ and, under the isometry, the mean curvature of $\p M$ in
$(M, g $) equals $H$; and

\vspace{.1cm}

\item[(iii)]  $(M, g)$ satisfies certain non-degeneracy condition that prevents $\m (M, g)$ from being arbitrarily small; for instance,
it is often required that $(M, g)$  contains no closed minimal surfaces (enclosing $\p M$), or $\p M$ is area outer-minimizing in $(M, g)$.
\end{itemize}
We refer interested readers to \cite{Anderson-Jauregui, Jauregui-18, McCormick-18, Wiygul1} 
and references therein for other variations in the definition of $\mb (\cdot)$. 

For an arbitrary pair $(\gamma, H)$, it is an interesting problem to construct and parametrize admissible extensions of $(\Sigma, \gamma, H)$ 
(see Problems 1 - 3 in \cite{B3}). 
In the horizon boundary case, i.e. $H = 0 $, under an assumption 
$ \lambda_1 ( -\Delta_\gamma + K_\gamma) > 0 $, where $\Delta_\gamma$ is the Laplacian on $(\Sigma, \gamma)$,
$K_\gamma$ is the Gauss curvature of $\gamma$, and $\lambda_1 $ is the first eigenvalue of $ - \Delta_\gamma + K_\gamma$, 
Mantoulidis and Schoen \cite{MS} constructed admissible extensions of $(\Sigma, \gamma, 0)$, whose ADM mass can be made
arbitrarily close to $ \sqrt{ \frac{ | \Sigma |_{\gamma} }{16 \pi} } $, where $ |\Sigma|_\gamma$ is the area of $(\Sigma, \gamma)$.
As a result, Mantoulidis and Schoen \cite{MS} showed 
\be \label{eq-est-H-0}
 \mb (\Sigma, \gamma, 0 ) \le \sqrt{ \frac{ | \Sigma |_{\gamma} }{16 \pi} }  .
\ee
Combined with the Riemannian Penrose inequality \cite{Bray01, HI01}, \eqref{eq-est-H-0} determined 
\be \label{eq-H-0}
 \mb (\Sigma, \gamma, 0 ) = \sqrt{ \frac{ | \Sigma |_{\gamma} }{16 \pi} }  .
\ee
Such a result was later extended by Chau and Martens \cite{ChauMartens20, ChauMartens22} to metrics $\gamma$ satisfying 
$  \lambda_1 ( -\Delta_\gamma + K_\gamma)  = 0 $. 

In the CMC boundary case, i.e. $H = H_o$ is a positive constant, there have been a sequence of works 
that adapted Mantoulidis-Schoen's methodology 
to derive upper bounds for $\mb(\Sigma, \gamma, H_o)$, see \cite{CCMM17, MWX18, ChauMartens22}.
Also in the CMC case, an earlier work of Lin and Sormani \cite{LS14} gave estimates 
of $\mb (\Sigma, \gamma, H_o)$ by 
using Ricci flow to construct admissible extensions of $(\Sigma, \gamma, H_o)$. 
In all these mentioned works, the metric $\gamma$ is assumed to have either positive or nonnegative Gauss curvature. 

If $\gamma$ has positive Gauss curvature and the mean curvature function $H$ is positive, Shi and Tam \cite{ST} constructed an admissible extension
of $(\Sigma, \gamma, H)$ based on earlier quasi-spherical metric constructions of Bartnik \cite{B93}. 
For such a pair $(\gamma, H)$,  Shi-Tam's result \cite{ST} yields  
\be \label{eq-ST-02}
\mb (\Sigma, \gamma, H) \le \frac{1}{8\pi} \int_{\Sigma} (H_0 - H) \, d \mu_\gamma , 
\ee
where $H_0 $ is the mean curvature of the isometric embedding of $(\Sigma, \gamma)$ in the Euclidean space $ \R^3$ and $ d \mu_\gamma$ denotes the area
form on $(\Sigma, \gamma)$.

In the special case $\gamma = \sigma_o$, a round metric on $\Sigma$, adapting the construction of Shi-Tam \cite{ST}, 
the first author \cite{M07, M07-p} derived a sharp upper bound of $\mb (\Sigma, \sigma_o, H)$ 
with $H$ being an arbitrary, positive function: 
\be \label{eq-M-07}
\mb(\Sigma, \sigma_o, H) \le \sqrt{ \frac{ | \Sigma |_{\gamma} }{ 16 \pi} } \left[ 1 - \frac{1}{16 \pi | \Sigma |_\gamma } \left( \int_\Sigma H \, d \mu_\gamma \right)^2 \right] .
\ee
Equality in \eqref{eq-M-07} holds if and only if the data $(\sigma_o, H)$ arises from CMC round spheres in spatial Schwarzschild manifolds.

In this work, we extend estimate \eqref{eq-M-07} to allow arbitrary metrics $\gamma$ with nonnegative Gauss curvature. Our main result is the following.

\begin{thm} \label{thm-main}
Let $ \gamma $ be a metric of nonnegative Gauss curvature on the two-sphere $ \Sigma$. 
Let $ H  $ be a positive function on $\Sigma$. 
Then the Bartnik mass $\m_B (\Sigma, \gamma, H)$ satisfies
\begin{equation} \label{eq-mB-intro}
\m_B (\Sigma, g, H) \leq  \sqrt{ \frac{ | \Sigma |_\gamma} { 16 \pi} } 
\left[ \left( 1 + \frac{ \zeta (\gamma)  }{ 8 \pi  r_\gamma }  \int_\Sigma H \, d \mu_\gamma \right)^2
- \left( \frac{  1 }{ 8 \pi  r_\gamma }  \int_\Sigma H \, d \mu_\gamma \right)^2   \right].
\end{equation}
Here $ r_\gamma = \left( \frac{ | \Sigma |_\gamma }{ 4 \pi } \right)^\frac12 $, 
$\zeta ( \gamma ) \ge 0 $ is a constant depending only on $\gamma$, and is 
invariant under scaling of $\gamma$.
If $\gamma$ has positive Gauss curvature, then
$\zeta (\gamma) \le C (\kappa)$  for some constant $ C(\kappa)$ depending only on 
$ \kappa = \displaystyle \frac{ \max_\Sigma K_\gamma } { \min_\Sigma K_\gamma} $.
Moreover, there exists a small  $\epsilon > 0 $, such that,  if 
$ \kappa < 1 + \epsilon $, then
\be \label{eq-zeta-kappa-1-intro}
 \zeta (\gamma) \le  C | \kappa - 1 |, 
\ee
where $C$ is some absolute constant.
\end{thm}

\begin{rmk}
If $\gamma$ has positive Gauss curvature, 
$(\Sigma, \gamma)$ isometrically embeds in $ \R^3$ as a convex surface
$\Sigma_0$ (\cite{Nirenberg, Pogorelov}). Let $H_0$ be the mean curvature of $\Sigma_0$ in $ \R^3$. 
The Riemannian positive mass theorem (\cite{SchoenYau79, Witten81}) shows
$$\mb(\Sigma_0, \gamma, H_0) = 0 . $$ 
Applying  Theorem \ref{thm-main} to $(\Sigma_0, \gamma, H_0) $, we then have
$$ 0  \le \sqrt{ \frac{ | \Sigma |_\gamma} { 16 \pi} } 
\left[ \left( 1 + \frac{ \zeta (\gamma)  }{ 8 \pi  r_\gamma }  \int_\Sigma H_0 \, d \mu_\gamma \right)^2
- \left( \frac{  1 }{ 8 \pi  r_\gamma }  \int_\Sigma H_0 \, d \mu_\gamma \right)^2   \right], $$
which translates into the following lower bound of $\zeta (\gamma)$: 
\be \label{eq-lower-bd-zeta}
\zeta (\gamma) \ge 1 - \frac{ 8 \pi r_\gamma}{ \int_{\Sigma_0} H_0 \, d \mu_\gamma } . 
\ee
By the classic Minkowski inequality, the right side of \eqref{eq-lower-bd-zeta} is $\ge 0 $ and is zero if and only 
$\Sigma_0$ is a round sphere. 
We think \eqref{eq-lower-bd-zeta} is an interesting lower bound on $\zeta (\gamma)$ because  
$\zeta (\gamma)$ is defined as the infimum over a family of quantities measuring the roundness of $\gamma$,
see the definition \eqref{eq-def-zeta}. 
It follows from \eqref{eq-lower-bd-zeta} and \eqref{eq-def-zeta} that 
$\zeta (\gamma) = 0 $ if and only if $\gamma$ is a round metric.
 \end{rmk}
 
 \begin{rmk}
If we denote the right side of \eqref{eq-mB-intro} by $\tilde \m (\Sigma, \gamma, H)$, it can be shown, along large 
coordinate spheres, $\tilde \m ( \Sigma, \gamma, H)$ approaches the mass of an asymptotically Schwarzschild manifold.
More precisely, suppose $(M, g)$ is a $3$-manifold such that, outside a compact set, $M$ is diffeomorphic to 
$ \R^3$ minus a ball and the metric coefficients $g_{ij}$ satisfies 
$$ g_{ij} = ( 1 + 2m r^{-1} ) \delta_{ij} + O ( |x|^{-2} ) , \ \text{as} \ x \to \infty, $$
where $m$ is a constant and equals the mass of $(M,g)$.
Let $ S_r = \{ | x | = r \}$, and let $\sigma_r$, $H_r$ denote the induced metric, the mean curvature of $S_r$ in $(M,g)$. 
Then, direct calculation gives
\begin{equation*} 
\begin{split}
\int_{S_r} H_r \, d \mu_{\sigma_r} = & \ 8 \pi r + O ( r^{-1} ) , \\
 K_{\sigma_r} = & \ r^{-2} ( 1 + 2m r^{-1} )^{-1}  + O (r^{-4})  , 
\end{split}
\end{equation*}
(for instance, see (5.10) and (5.14) in \cite{ST}). The equation on $K_{\sigma_r}$ implies the curvature ratio $\kappa$, associated to $\sigma_r$, 
 satisfies $ \kappa = 1 + O ( r^{-2} )$. Hence, $ \zeta (\sigma_r) = O ( r^{-2} )$ by \eqref{eq-zeta-kappa-1-intro}.
 These,  together with  the fact $ | S_r |_{\sigma_r} = 4 \pi r^2 ( 1 + 2 m r^{-1} ) ( 1 + O ( r^{-2} ) )$,
  readily implies 
 \be
  \tilde m ( S_r, \sigma_r, H_r) \to m , \ \text{as} \ r \to \infty . 
 \ee 
 \end{rmk}

\vspace{.2cm}

In \cite{MX19}, Xie and the first author found the Mantoulidis-Schoen estimate \eqref{eq-est-H-0}, in the case of metrics $\gamma$ with positive Gauss curvature, can be reproduced by combing the methods in \cite{ST} and \cite{M07}. Moreover, in \cite{MX19} it was shown
\be \label{eq-est-MX}
\mb (\Sigma, \gamma, H) \le \sqrt{ \frac{ | \Sigma |_\gamma}{16 \pi } }
\ee
for any positive function $H$. 

Our derivation of Theorem \ref{thm-main} is motivated by the work in \cite{MX19}.
Briefly speaking, one starts with a special path of metrics $\{\gamma(t) \}_{t \in [0,1]}$, constructed by Mantoulidis-Schoen, which
connects the given metric $\gamma$ to a round metric $\sigma_o$. 
Upon reparameterizing and suitably scaling $\{ \gamma (t) \}$, one obtains a path of metrics 
$\{ \bar \gamma_s \}_{s \in [1, \infty) }$. On the product manifold $\Sigma \times [1, \infty)$ with
a background metric $\bar g = d s^2 + \bar \gamma_s $, one then performs a Bartnik-Shi-Tam type construction 
to build an admissible extension of $(\Sigma, \gamma, H)$.
Carefully tracing how the total mean curvature evolves  along the foliation in the extension, one can relate the mass
of the extension to the total mean curvature at the initial surface as well as the ``expense" paid by connecting 
$\gamma$ to a round metric. The area radius appears in the estimate as a normalization factor. 

Besides \eqref{eq-mB-intro}, estimates in this paper also give an extension of \eqref{eq-est-MX} to metrics $\gamma$ with
nonnegative Gauss curvature, see Corollary \ref{cor-gen-MX}. 

In a suitable sense, a dual problem of estimating the Bartnik mass is a problem of estimating the supremum of the total mean
curvature of nonnegative scalar curvature (NNSC) fill-ins of a given close manifold.  
Interested readers are referred to \cite{Jauregui13, MM16, Gromov18, Gromov19, SWWZ, SWW20, M20}
for results and questions related to NNSC fill-ins. 
As a byproduct in this work, we observe a lower bound of the supremum of the total mean
curvature of NNSC fill-ins of a given manifold with positive scalar curvature, see Theorem \ref{thm-Lambda}.

\section{Extensions and mass estimates} \label{sec-deformation} 

Let $\Sigma $ denote an $(n-1)$-dimensional sphere, $ n \ge 3$.
Let $\gamma$ be a Riemannian metric with nonnegative scalar curvature on $ \Sigma$. 
 Let $r_\gamma$ be the volume radius of $(\Sigma, \gamma)$, 
i.e. $|\Sigma|_\gamma = \omega_{n-1}  r_\gamma^{n-1}$,  where  $\omega_{n-1}$ is the volume 
of a round sphere of radius one in $ \R^{n}$.

Suppose $\{ \gamma (t) \}_{ t \in [0,1]} $ is a smooth path of metrics on $\Sigma$ satisfying the following properties:
\begin{itemize}
\item[i)]  $\gamma(0)=\gamma$, $\gamma(1)$ is a round metric with the same volume as $\gamma$;   
\item[ii)] $ \gamma (t) $ has positive scalar curvature for $ t > 0 $; 
\item[iii)]  $\tr_{\gamma(t) } \gamma' (t) =0$ for $ t \ge 0 $, where $\gamma'(t) = \frac{d}{dt} \gamma (t)$.
\end{itemize}
We will comment on condition iii) in Section \ref{sec-uniformization}. For the moment, fix such a path 
and define
\begin{equation}\label{def-alpha(t)-beta(t)}
\alpha(t) =\frac{\max_{\Sigma}\left|\frac{1}{2}  \gamma'(t)\right|_{\gamma(t)}^2}{n-1} ,\quad \beta(t)=\dfrac{r_\gamma^2\min_\Sigma R_{\gamma(t)}}{(n-1)(n-2)}.
\end{equation}
Here $R_{\gamma(t)}$ denotes the scalar curvature of $\gamma(t)$. Note that $\alpha(t)$ and $\beta(t)$ are scaling invariant 
in the sense, if $\{ \gamma (t) \}$ is replaced by $ \{ c^2 \gamma (t) \}$ for a constant $c>0$,  $\alpha (t)$ and $\beta(t)$ 
will remain unchanged. 

Next, we adopt a construction from \cite{MX19}.
Given a constant $ \ts  > 1$,  consider a smooth function
\begin{equation}\label{rep}
t (\cdot) :[1,\infty)\rightarrow [0,1],  \ \ \text{with} \ 
t(1)= 0 \ \text{and} \  t(s)=1, \,\,\forall \, s\geq \ts.
\end{equation}
For each $ s \in [1, \infty) $, define
$$ \gamma_s = r_{\gamma}^{-2} \gamma(t(s)) .$$
$\{ \gamma_s \}_{s \ge 1}$ satisfies $ \gamma_1  = r_\gamma^{-2} \gamma$ 
and $\gamma_s  = \sigma_o $,  $ s \ge \ts$, 
where $ \sigma_o$ is a round metric on $ \Sigma$ with volume $\omega_{n-1}$.
Let $M=[1,\infty)\times \Sigma$ and $ \Sigma_s =  \{ s \} \times \Sigma $. 
On $M$, consider a background metric
\begin{equation*}
\bar{g}=ds^2+\bar\gamma_s, \ \ \text{where} \ \bar\gamma_s=s^2\gamma_s. 
\end{equation*}
This metric $\bar g$ has the following features: 
\begin{itemize}
\item[a)] the induced metric $\bar \gamma_s$ on $ \Sigma_s  $ has positive scalar curvature for $ s > 1$; 
\item[b)] the second fundamental form $\bar A_s$ and the mean curvature $ \bar H_s$ of $\Sigma_s$ in $(M, \bar g)$
satisfy
\begin{equation} \label{As-Hs}
\bar{A}_s =\frac{\bar{\gamma}_s}{s}+\frac{1}{2}s^2\gamma'_s \ \ \text{and} \ \
\bar{H}_s =\frac{n-1}{s}, \ \forall \, s \ge 1. 
\end{equation}
Here $ \gamma'_s = \frac{d}{d s} \gamma_s $ and condition iii) is used in obtaining $ \bar H_s = \frac{n-1}{s} $.

\vspace{.1cm}

\item[c)]  $\bar g = d s^2 + s^2 \sigma_o $ is a Euclidean metric on $ (\ts , \infty) \times \Sigma$.
\end{itemize}

\vspace{.2cm}

The following lemma follows directly from results in \cite{ST, EMW}. 

\begin{lem}\label{lem-quasi-spherical}
Given any positive function $ H > 0$ on $ \Sigma$, 
there exists a positive function $u$ on $M$ so that
\begin{enumerate}
\item[1)] $g=u^2ds^2+\bar{\gamma}_s$ has zero scalar curvature;
\item[2)] the mean curvature $H_1$  of $\Sigma_1= \partial M $ in $(M, g)$ equals $ r_{\gamma}  H$; 
\item[3)]  $u\rightarrow 1$ as $s\rightarrow \infty$ and $(M, g) $ is asymptotically flat, foliated 
by $\{ \Sigma_s \}_{s \ge 1}$ with positive mean curvature. 
\end{enumerate}
\end{lem}

\begin{proof}
Let $ \Delta_s $ denote the Laplacian on $(\Sigma_s, \bar \gamma_s)$.
The equation on $u$ corresponding to conditions 1) and 2) is 
\be \label{eq-u-PDE}
\left\{
\begin{split}
\frac{\p }{\p s} u  = & \ \frac{1}{\bar{H}_s} u^2 \Delta_s u + \frac{u}{2 \bar{H}_s} 
\left( \bar{H}_s^2 + | \bar{A}_s |^2 + 2 \p_s \bar{H}_s \right)  
- u^3 \frac{ K_{\bar{\gamma}_s} }{\bar{H}_s} , \ \ s \ge 1,  \\
u|_{s=1}  = & \frac{1}{ r_\gamma H} \bar{H}_1 
\end{split}
\right.
\ee
(see equation (5) in \cite{EMW} for instance).
Since $\bar{H}_s > 0$, \eqref{eq-u-PDE} has a positive solution on some small interval 
$[1, 1 + \delta)$, $\delta > 0 $. 
Since $ K_{\bar{\gamma}_s} > 0 $ for $ s > 0$, the solution exists on $[1, \infty)$ by 
\cite[Proposition 2]{EMW}. Since $\bar g$ is the Euclidean metric on $(b, \infty) \times \Sigma$, 
the claim that $u$ satisfies 3) follows \cite[Theorem 2.1]{ST}.
\end{proof}

Let $g$ be the metric given in Lemma \ref{lem-quasi-spherical} and let $ \m (g)$ denote its mass.  
Let $ H_s $ be the mean curvature of $ \Sigma_s$ in $(M, g)$.
Define 
$$ \mathcal{H}_s=  \frac{1}{ (n-1) \omega_{n-1} } \int_{\Sigma_s} H_s \, d \mu_s ,$$
where $d \mu_s$ is the volume form on $ (\Sigma_s,  \bar \gamma_s)$. 
As $ \bar g $ is a Euclidean metric on $ (\ts , \infty) \times \Sigma$, 
we apply  \cite[Theorem 2.1]{ST} to deduce 
\begin{equation}
\int_{\Sigma_s}  \bar H_s \, d \mu_s - \int_{\Sigma_s} H_s  \, d \mu_s  = (n-1) \omega_{n-1} \m (g) + o (1), \ \text{as} \ s \to \infty. 
\end{equation}
Since $\Sigma_s$ is a round sphere of radius $s$ in $ \left( (\ts , \infty) \times \Sigma, \bar g \right)$, 
$$ \int_{\Sigma_s} \bar H_s \, d \mu_s = (n-1) \omega_{n-1} s^{n-2}, \ \forall \, s > \ts . $$
Thus,  
\begin{equation} \label{eq-Hs-expansion}
\mathcal{H}_s = s^{n-2}  - \m (g) + o (1), \ \text{as} \ s \to \infty. 
\end{equation} 

\vspace{.2cm}

We analyze how $\mathcal{H}_s$ evolves along $\{ \Sigma_s \}$. 
The next proposition was inspired by 
a computation of Shi-Wang-Wei-Zhu \cite[page 249]{SWWZ}.

\begin{prop}\label{prop-mean curv} 
The total mean curvature $\mathcal{H}_s$ satisfies 
\begin{equation}\label{eq-mean-curv-der}
\frac{d\mathcal{H}^2_s}{ds}\geq  \left(\frac{n-2}{s}-\alpha_s|t'(s)|^2{s}\right)\mathcal{H}^2_s 
+ (n-2) s^{2n-5} \beta_s, \ \ \forall \, s \ge 1 .
\end{equation}
Here $\alpha_s=\alpha(t(s))$ and $ \beta_s=\beta(t(s))$.
\end{prop}

\begin{proof}
 By the second variation of volume and the Gauss equation,  
\begin{equation*}
\frac{\partial }{\partial s} H_s =  \frac{1}{2} R_{ \bar{\gamma}_s }   u - \Delta_s u -  \frac{1}{2} \left( | \bar A_s|_{\bar \gamma_s} ^2 +  \bar H_s^2 \right) u^{-1} .
\end{equation*}
Thus, 
\begin{equation}\label{main1}
\begin{aligned}
\dfrac{d}{ds}\int_{\Sigma_s}  H_s \, d\mu_s =\frac{1}{2}\int (\bar{H}_s^2- | \bar A_s|_{\bar \gamma_s} ^2)u^{-1} \, d\mu_s+\frac{1}{2}\int R_{ \bar{\gamma}_s}   u \, d\mu_s.
\end{aligned}
\end{equation}
By \eqref{As-Hs} and \eqref{def-alpha(t)-beta(t)}, 
\begin{equation} \label{roundness}
\left|\bar{A}_s-\frac{\bar{\gamma_s}}{s}\right|^2_{\bar{\gamma}_s}
=\left|\frac{1}{2} \frac{d \gamma (t) }{d t} \, t'(s)  \right|^2_{\gamma (t) }\leq (n-1)\alpha_s \, |t'(s)|^2=\alpha_s \, |t'(s)|^2 s\bar{H_s} .
\end{equation}
It follows from \eqref{main1} and \eqref{roundness} that
\begin{equation}\label{main2}
\begin{split}
\dfrac{d}{ds}\int_{\Sigma_s}  H_s d\mu_s & \ 
= \frac{1}{2}\int_{\Sigma_s}  \left(\frac{n-2}{s}\bar{H}_s-\left|\bar{A}_s-\frac{\bar{\gamma}}{s}\right|^2_{\bar{\gamma}_s}
\right) u^{-1} \, d\mu_s+\frac{1}{2}\int_{\Sigma_s}  R_{ \bar{\gamma}_s}  u  \, d\mu_s\\
& \ \geq \frac{1}{2} \left(\frac{n-2}{s}  -\alpha_s|t'(s)|^2s \right)  \int_{\Sigma_s}  \bar{H}_s \, u^{-1} \, d \mu_s \\
& \ +\frac{(n-1)(n-2) \beta_s }{2s^2}\int_{\Sigma_s}  u \, d\mu_s .
\end{split}
\end{equation}
By H\"{o}lder's inequality, 
$$\int_{\Sigma_s}  u \, d\mu_s \geq \dfrac{ | \Sigma_s|_{\bar \gamma_s}^2}{\int_{\Sigma_s}  u^{-1} \, d\mu_s}
=\frac{ ( s^{n-1} \omega_{n-1} )^2}{\int_{\Sigma_s}  u^{-1} \, d\mu_s}.$$
Hence, \eqref{main2} and \eqref{As-Hs} imply 
\begin{equation}
\begin{split}
\frac{d }{ds} \int_{\Sigma_s} H_s \, d \mu_s 
\geq & \ \frac{1}{2} \left(\frac{n-2}{s}-\alpha_s|t'(s)|^2{s}\right) \int_{\Sigma_s} H_s \, d \mu_s \\
& \ + \frac{(n-1)(n-2) \beta_s}{2s^2 }
\frac{s^{2n-2}\omega_{n-1}^2(n-1)}{s \int_{\Sigma_s} H_s \, d \mu_s } ,
\end{split} 
\end{equation}
which proves \eqref{eq-mean-curv-der}.
\end{proof}

\begin{rmk}
If $ \gamma $ is a round metric on $ \Sigma$, one can take 
$ \{ \gamma (t) \}_{t \in [0, 1] }$ to be a constant path of metrics. In this case, $ \alpha (t) = 0 $, $ \beta (t) = 1$, 
and \eqref{eq-mean-curv-der} becomes
\begin{equation*} 
\frac{d\mathcal{H}^2_s}{ds}\geq \frac{n-2}{s} \mathcal{H}^2_s+ (n-2)  s^{2n-5} ,
\end{equation*}
or equivalently 
\begin{equation}
\frac{d}{d s} \left\{ \left( \frac{ | \Sigma_s |_{\bar \gamma_s}  } { \omega_{n-1} } \right)^\frac{n-2}{n-1} \left[ 1 - \left( \frac{ | \Sigma_s |_{\bar \gamma_s}  } { \omega_{n-1} } \right)^\frac{2 (2-n) }{n-1} \mathcal{H}_s^2 \right] \right\} \le 0 . 
 \end{equation} 
This monotone property gives another insight into \cite[Theorem 1]{M07}.
\end{rmk}

\begin{rmk}
In deriving \eqref{eq-mean-curv-der}, one does not need $\gamma (1)$ to be a round metric; neither does $\Sigma$ need to be a
sphere. We will explore this fact in Section \ref{sec-NNSC}.
\end{rmk}

\vspace{.2cm}

In the rest of this section, we focus on the dimension $n=3$. 
In this case, $\gamma$ is a metric with nonnegative Gauss curvature $K_\gamma$ on the $2$-sphere $ \Sigma$. 
By Gauss-Bonnet theorem, 
$$ \beta(t) \le 1.$$ 
For convenience, we normalize $\gamma$ so that $ | \Sigma |_\gamma = 4 \pi$, i.e. $ r_\gamma = 1$.

Choosing $ n = 3 $ in Proposition \ref{prop-mean curv}, we have
\begin{equation}
\frac{d\mathcal{H}^2_s}{ds}- \left(\frac{1}{s}-\alpha_s|t'(s)|^2{s}\right)\mathcal{H}^2_s\geq  s \beta_s ,
\end{equation}
which implies 
\begin{equation} \label{eq-int}
\frac{d}{ds}\left( s^{-1} e^{\int_1^s \alpha_s|t'(s)|^2{s} \, ds} \,  \Hs \right)\geq  \beta_s \, e^{\int_1^s \alpha_s|t'(s)|^2{s} \, ds}.
\end{equation}
Integrating \eqref{eq-int} from 1 to $s>\ts$, we have
\begin{equation} \label{eq-ineq-Hs}
\begin{split}
& \  \frac{e^{\int_1^{\ts} \alpha_s|t'(s)|^2{s} ds}  }{ s} \Hs - \Ho \\
\geq & \ \int_1^{\ts} \beta_s e^{\int_1^s \alpha_s|t'(s)|^2{s} \, ds} \, ds + (s-\ts)e^{\int_1^{\ts} \alpha_s|t'(s)|^2 s \, ds} .
\end{split}
\end{equation}
Here we used the fact $ \gamma_s = \sigma_o $, hence $ \alpha_s = 0 $ and $ \beta_s = 1$, $\forall \, s \ge \ts $.

Re-writing \eqref{eq-ineq-Hs} as  
\begin{equation}\label{eq-inequality-Hs}
\frac{ \Hs }{ s}- e^{-\int_1^{\ts} \alpha_s|t'(s)|^2{s} \, ds} \, \Ho \geq (s-\ts)
+ \int_1^{\ts} \beta_s e^{ - \int_{s}^\ts  \alpha_s|t'(s)|^2{s} \, ds} \, ds ,
\end{equation}
letting $ s \to \infty$ and applying \eqref{eq-Hs-expansion},
we obtain
\begin{equation} \label{eq-est-mg}
2 \m (g) \le   \ts  -  \int_1^{\ts} \beta_s e^{ - \int^{\ts}_s  \alpha_s|t'(s)|^2{s} \, ds} \, ds - e^{-\int_1^{\ts} \alpha_s|t'(s)|^2{s} \, ds} \,  \Ho .
\end{equation} 

By Lemma \ref{lem-quasi-spherical}, $(M, g)$ is an asymptotically flat extension of $(\Sigma, \gamma, H)$, $(M, g)$ has zero 
scalar curvature and is foliated by positive mean curvature surfaces $\{ \Sigma_s \}$.
Hence, by \eqref{eq-est-mg} and the definition of the Bartnik mass,
\begin{equation} \label{eq-mB-1}
\begin{split}
\m_B (\Sigma, g, H) \le & \ \m (g) \\
\le & \  \frac12 \left[ \ts -  \int_1^{\ts} \beta_s e^{ - \int^{\ts}_s  \alpha_s|t'(s)|^2{s} \, ds} \, ds 
 - e^{-\int_1^{\ts} \alpha_s|t'(s)|^2{s} \, ds} \,  \Ho \right].
\end{split}
\end{equation}
In general, when $(\Sigma, \gamma)$ does not necessarily have area $4\pi$, it is easily checked 
\begin{equation} \label{eq-mB-scaling} 
\m (\Sigma, g, H)  = r_\gamma \, \m (\Sigma, r_\gamma^{-2} \gamma, r_\gamma H ) .
\end{equation}
The following proposition follows from \eqref{eq-mB-1} and \eqref{eq-mB-scaling}. 

\begin{prop} \label{prop-est-1}
Let $ \gamma $ be a metric with nonnegative Gauss curvature on the two-sphere $ \Sigma$. 
Let $ H  $ be a positive function on $\Sigma$. 
Suppose $ \{ \gamma (t) \}_{ t \in [0, 1] }$ is a path of metrics satisfying i), ii) and iii). 
Given any constant $ b > 1$ and any $C^1$ function 
\begin{equation} \label{rep-1}
 t( \cdot) : [1, b] \rightarrow [0, 1] \ \text{with} \ t(1) = 0 \ \text{and} \ t (b) = 1 , 
\end{equation} 
the Bartink mass $\m_B (\Sigma, \gamma, H)$ satisfies
\begin{equation} \label{eq-mB-2}
\m_B (\Sigma, g, H) \leq  \frac{r_\gamma}{2} \left[ \ts  -  \int_1^{\ts} \beta_s e^{ - \int^{\ts}_s  \alpha_s|t'(s)|^2{s} \, ds} \, ds 
- e^{-\int_1^{\ts} \alpha_s|t'(s)|^2{s} \, ds}  \, \mathcal{H}^2 \right].
\end{equation}
where $ r_\gamma = \sqrt{ \frac{ | \Sigma |_\gamma} { 4 \pi} } $ 
and $ \mathcal{H} =  \frac{ 1}{ 8 \pi  r_\gamma } \int_\Sigma H \, d \mu_\gamma $. 
\end{prop}

\begin{rmk}
We comment on the $C^1$ assumption on the function $t(s)$ in \eqref{eq-mB-2}. 
The argument preceding \eqref{eq-mB-1} readily shows \eqref{eq-mB-2} holds for any $ t(\cdot): [1, b] \rightarrow [0, 1]$ which 
is the restriction of a smooth function $ t(\cdot) : [1, \infty) \rightarrow [0,1]$ satisfying \eqref{rep}.
Now if $ t(\cdot) $ is merely $C^1$ on $[0,1]$ satisfying \eqref{rep-1}, 
one can consider a function $T(s)$ on $[1,\infty)$ so that $ T (s) = t(s)$ on $[1,\ts]$ and $T(s)=1$ for $s>\ts$. 
$T(s)$ may not be smooth on $[1, \infty)$, but one can mollify it. 
For instance, one can first compress the graph of $ T(s)$  horizontally by a small factor $\epsilon>0$
and denote such a function by $T_\epsilon (s)$.
Then one can smooth $ T_\epsilon $ out via a usual mollifier $\phi_{\epsilon_1}>0$, with $\epsilon_1>0$ small enough 
so that $ T_\epsilon \ast \phi_{\epsilon_1}  =   T_\epsilon $  near the points $1$ and $b$ (see Figure 1). 
As the right side of \eqref{eq-mB-1} depends on $ t(\cdot) $ only via $ t'(s)$ on $[1, b]$, 
 letting  $\epsilon_1$ and $ \epsilon$ tend to $0$,  one obtains Proposition \ref{prop-est-1}.
\end{rmk}

\begin{center}

\begin{figure}
\begin{tikzpicture}[
    scale=1.8,
    axis/.style={thick, ->, >=stealth'},
    important line/.style={thick},
    dashed line/.style={dashed, thick},
    every node/.style={color=black,},
    use Hobby shortcut
 ]

\draw[axis] (0,0)  -- (3,0) node(xline)[right] {$s$};
\foreach \x/\xtext in {.5/$1$,2/$\ts$}{
      \draw (\x cm,-1pt) -- (\x cm,1pt) node[below,pos=-.01pt] {\xtext};
    }
\foreach \y in {1}
    \draw (1pt,\y cm) -- (-1pt,\y cm) node[anchor=east] {$\y$};     

\draw[axis] (0,0) -- (0,1.3) node(yline)[above] {$t$};
\draw[dashed] (2,0)--(2,1);

\draw (2,1.2) node{$T(s)$};
\draw[very thick] ([out angle=55].5,0) .. (2,1);
\draw[very thick] (2,1) .. (3,1);

\end{tikzpicture}
\hspace{1cm}
 \begin{tikzpicture}   [
    scale=1.8,
    axis/.style={thick, ->, >=stealth'},
    important line/.style={thick},
    dashed line/.style={dashed, thick},
    every node/.style={color=black,},
    use Hobby shortcut
 ]
  
      \draw[axis] (0,0)  -- (3,0) node(xline)[right] {$s$};
\foreach \x/\xtext in {.5/1,.7/\hspace{.6cm}$1\hspace{-.1cm}+\hspace{-.1cm}\epsilon$,1.8/\hspace{-.5cm}$\ts  \hspace{-.1cm}-\hspace{-.1cm}\epsilon$,2/\hspace{.1cm}$\ts$}{
      \draw (\x cm,-1pt) -- (\x cm,1pt) node[below,pos=-.01pt] {\xtext};
    }
\foreach \y in {1}
    \draw (1pt,\y cm) -- (-1pt,\y cm) node[anchor=east] {$\y$};     

\draw[axis] (0,0) -- (0,1.3) node(yline)[above] {$t$};
\draw[dashed] (2,0)--(2,1);

\draw (1.89,.75) node[text=red]{$\phi_{\epsilon_1}\hspace{-.15cm}\ast T_\epsilon$};
\draw (2,1.2) node[text=blue]{$T_\epsilon(s)$};
\draw[red, very thick] ([out angle=5].5,0) .. (.857,.23);
\draw[red, very thick] (1.46,.8) .. ([out angle=35]2,1);
\draw[blue, very thick] ([out angle=58].7,0) .. (1.8,1);
\draw[blue, very thick] (1.79,1) .. (3,1);
\end{tikzpicture}

 \caption{ On the left is the graph of $T(s)$, with a corner at $\ts$. On the right in blue is the compression $T_\epsilon(s)$ of $T(s)$ and in red its smoothing using a mollifier $\phi_{\epsilon_1}$ that preserves the endpoints.}
       
\label{fig1} 

\end{figure}
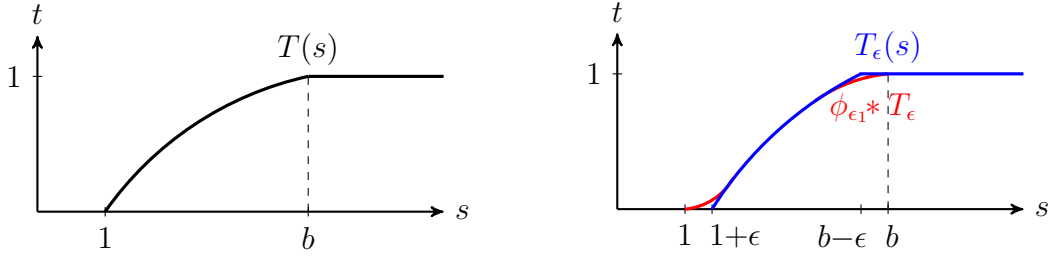

\end{center}

In Proposition \ref{prop-est-1}, the function $ t(s)$ does not need to be monotone. 
If $t(s)$ is chosen to be monotone,
\eqref{eq-mB-2} can be reformulated in terms of the inverse function $s=s(t)$. 
More precisely, Proposition \ref{prop-est-1} shows 

\begin{prop} \label{prop-est-2}
Let $ \gamma $ be a metric with nonnegative Gauss curvature on the two-sphere $ \Sigma$. 
Let $ H  $ be a positive function on $\Sigma$. 
Suppose $ \{ \gamma (t) \}_{ t \in [0, 1] }$ is a path of metrics satisfying i), ii) and iii). 
Given  any  $C^1$ function 
\begin{equation} \label{eq-condition-st}
 s = s (t), \,  t \in [0, 1], \, \text{with} \ s (0) = 1 \ \text{and} \ s'(t) > 0, 
\end{equation}
the Bartink mass $\m_B (\Sigma, \gamma, H)$ satisfies
\begin{equation} \label{eq-mB-3}
\begin{split}
\m_B (\Sigma, g, H) \leq & \  \frac{ r_\gamma}{2} 
\left[ s(1)  -  \int_0^{1}  \frac{ \beta (t)   s'(t)  }{ e^{ \int^{1}_t   \alpha(t)  \frac{s(t)}{s'(t)}  \, dt} }  \, d t 
- \frac{1}{e^{  \int_0^{1} \alpha(t)  \frac{s(t)}{s'(t)}  \, dt}  } \, \mathcal{H}^2 \right] ,
\end{split}
\end{equation}
where $ r_\gamma = \sqrt{ \frac{ | \Sigma |_\gamma} { 4 \pi} } $ 
and $ \mathcal{H} =  \frac{ 1}{ 8 \pi  r_\gamma } \int_\Sigma H \, d \mu_\gamma $. 
\end{prop}

Given any $C^0$ function $ \phi (t) > 0 $ on $[0,1]$ and any constant $ k > 0$, plugging in \eqref{eq-mB-3} a choice of 
$$ s (t) = 1 + k \int_0^t \phi (t) d t  , $$
one has
\begin{equation} \label{eq-mB-5}
\begin{split}
\m_B (\Sigma, g, H) \leq & \  \frac{r_\gamma}{2} 
\left[ 1 + k \int_0^1 \left( 1 -  \frac{ \beta (t)      }{ e^{ \int^{1}_t   \alpha(t)  \left( \frac{ 1 }{ k \phi (t) } 
+ \frac{ \int_0^t \phi (t) d t }{  \phi (t) }  \right)   \, dt} } \right) \phi(t) \,  d t  \right. \\
& \ \left. 
- \frac{1}{e^{ \int^{1}_0   \alpha(t)  \left( \frac{ 1 }{ k \phi (t) } 
+ \frac{ \int_0^t \phi (t) d t }{  \phi (t) }  \right)   \, dt}  } \, \mathcal{H}^2 \right].
\end{split}
\end{equation}
Letting $ k \to 0 +$ in \eqref{eq-mB-5} gives the following corollary:

\begin{cor} \label{cor-gen-MX}
Let $ \gamma $ be a metric of nonnegative Gauss curvature on the $2$-sphere $ \Sigma$. 
Let $ H  $ be a positive function on $\Sigma$.  Then
\be
\m_B (\Sigma, g, H) \leq  \frac{r_\gamma}{2} .
\ee
\end{cor}

\vspace{.2cm}

For a fixed path $\{ \gamma (t) \}_{t \in [0,1]}$, an optimal estimate on $ \m_B(\cdot)$ from
Proposition \ref{prop-est-2} would be obtained by minimizing the right side of \eqref{eq-mB-3} over all $C^1$ function 
$ s(t)$ satisfying \eqref{eq-condition-st}. At the moment, we do not know a formula of such an infimum. Below, we proceed 
using an ad hoc ODE method to pick a choice of $s(t)$.

Suggested by 
$$ \left( e^{ - \int^{1}_t   \alpha(t)  \frac{s(t)}{s'(t)}  \, dt} \right)' =  e^{ - \int^{1}_t   \alpha(t)  \frac{s(t)}{s'(t)}  \, dt} \, \alpha(t) s(t) s'(t)^{-1} , $$
we choose $s(t)$ so that 
\begin{equation}\label{eq-ode-st}
\beta (t) s'(t) = k^2 \alpha(t) s(t) s'(t)^{-1}  ,
\end{equation} 
where $ k > 0 $ is an arbitrary constant. Clearly, \eqref{eq-ode-st} and $ s(0) = 1 $ shows
\begin{equation}
s (t) = \left( 1 + k \int_0^t \sqrt{ \frac{\alpha (t) }{ 4 \beta (t) } } \, d t \right)^2 , \ \ \text{if} \ 
{ \int_0^1 \sqrt{ \frac{\alpha (t)}{ \beta (t) }  } < \infty }.
\end{equation}
With this choice of $s(t)$, we have 
\be \label{eq-to-be-simplified}
\begin{split}
& \ s(1)  -  \int_0^{1}  \frac{ \beta (t)   s'(t)  }{ e^{ \int^{1}_t   \alpha(t)  \frac{s(t)}{s'(t)}  \, dt} }  \, d t 
- \frac{1}{e^{  \int_0^{1} \alpha(t)  \frac{s(t)}{s'(t)}  \, dt}  } \, \frac{  \left( \int_\Sigma H \, d \mu_\gamma \right)^2 }{16 \pi | \Sigma |_\gamma } \\
= & \ \left( 1 + k \int_0^1 \sqrt{ \frac{\alpha (t) }{ 4 \beta (t) } } \, d t \right)^2
- k^2  + e^{ - \int^{1}_0  \alpha(t)  \frac{s(t)}{s'(t)}  \, dt}  
\left[ k^2 -   \frac{ \left(  \int_\Sigma H \, d \mu_\gamma \right)^2 }{16 \pi | \Sigma |_\gamma } \right] .
\end{split}
\ee
To simplify the above quantity,  we may choose $ k $ so that 
\be
k =  \frac{  1 }{ 8 \pi  r_\gamma }  \int_\Sigma H \, d \mu_\gamma .
\ee

Thus, the following corollary follows from Proposition \ref{prop-est-2} and a choice of
\be
s (t) = \left( 1 + \frac{  \int_0^t \sqrt{ \frac{\alpha (t) }{ 4 \beta (t) } } \, d t  }{ 8 \pi  r_\gamma }  \int_\Sigma H \, d \mu_\gamma \right)^2 .
\ee
 
\begin{cor} \label{cor-est-1}
Let $ \gamma $ be a metric with nonnegative Gauss curvature on the two-sphere $ \Sigma$. 
Let $ H  $ be a positive function on $\Sigma$. 
Suppose $ \{ \gamma (t) \}_{ t \in [0, 1] }$ is a path of metrics satisfying i), ii) and iii). 
Then the Bartink mass $\m_B (\Sigma, \gamma, H)$ satisfies
\begin{equation} \label{eq-mB-4}
\m_B (\Sigma, g, H) \leq  \sqrt{ \frac{ | \Sigma |_\gamma} { 16 \pi} } 
\left[ \left( 1 + \frac{ \zeta (\gamma)  }{ 8 \pi  r_\gamma }  \int_\Sigma H \, d \mu_\gamma \right)^2
- \left( \frac{  1 }{ 8 \pi  r_\gamma }  \int_\Sigma H \, d \mu_\gamma \right)^2   \right].
\end{equation}
Here $\zeta ( \gamma ) \ge 0 $ is a scaling invariant quantity of $\gamma$, given by
\be \label{eq-def-zeta}
\zeta (\gamma) = \inf_{ \{ \gamma (t) \}_{t\in[0,1] } } \int_0^1 \sqrt{ \frac{\alpha (t) }{ 4 \beta (t) } } \, d t  .
\ee
\end{cor}

\section{Estimates of $\zeta (\gamma) $}\label{sec-uniformization}

Given a metric $\gamma$ on a two-sphere $\Sigma$, by the uniformization theorem,  
there exists a function $\varphi$ such that
\begin{equation}  \label{eq-uniformization-1}
r_\gamma^{-2} \gamma =e^{2\varphi}\sigma_o.
\end{equation}
Here $\sigma_o$ is a round metric on $ \Sigma$ with area $|\Sigma|_{\sigma_0}=4\pi$ and $ r_\gamma$ is the area
radius of $\gamma$.  
Let $ d \mu_o$ denote the area form of $\sigma_o$. Then 
$ 
\int_\Sigma e^{2\varphi}d\mu_0=4\pi.
$ 
In particular, $\varphi$ satisfies 
\begin{equation}\label{eq-min-conformal}
\min_{\Sigma} e^{2 \varphi} \le 1  \ \ \text{and} \ \ \min_{\Sigma} e^{-2 \varphi}\le 1     .
\end{equation}
As a result, there exists some point $p \in \Sigma$ so that $\varphi(p)=0$. 
Consequently, 
\begin{equation}\label{eq-bound-C0-norm-with-C1}
\norm{\varphi}_0\leq C_1 \norm{d\varphi}_0 ,
\end{equation}
where $C_1 $ is an absolute constant and $\norm{\cdot}_{0}$ denotes 
the $C^0$-norm of tensors on $(\Sigma, \sigma_o)$. 
Similarly, given any $\alpha \in (0,1)$, 
if  $\lsm \varphi \rsm_\alpha$ denotes a H\"{o}lder semi-norm of $ \varphi$ given by 
\begin{equation*}
\lsm \varphi\rsm_\alpha = \sup_{x, y \in \Sigma, \, x \neq y} \frac{|\varphi(x)-\varphi(y)|}{ d (x, y)^\alpha},
\end{equation*}
where $ d (x, y)$ is the distance on $(\Sigma, \sigma_o)$,
then
\begin{equation}\label{eq-bound-alpha-sn}
\lsm \varphi\rsm_\alpha\leq C_2 \norm{d\varphi}_{0},
\end{equation}
for some absolute constant $C_2$.   

The next proposition gives an estimate of $\zeta (\gamma)$ in terms of $\varphi$. 

\begin{prop} \label{prop-zeta-1}
Let $ \gamma $ be a metric with nonnegative Gauss curvature on a $2$-sphere $ \Sigma$. 
Let $ \zeta (\gamma)$ be given in \eqref{eq-def-zeta}. 
Let $ \varphi $ be a conformal factor in \eqref{eq-uniformization-1}. Then
\be \label{eq-est-zeta-1}
\zeta (\gamma)  \leq
C    e^{6 \norm{\varphi}_{0}}     \norm{\varphi}_{0, \alpha}   \left( 1 +  \norm{\varphi}_{0, \alpha}    \right)
  \left( 1 + \norm{d\varphi}_0 \right)    .
\ee
Here $ C $ is some constant depending on $\alpha$. 
\end{prop}

\begin{proof}
As in \cite{Nirenberg, MS}, a smooth path of metrics $ \{ \sigma (t) \}_{t \in [0, 1] }$ with constant area $4\pi$, connecting $ r_\gamma^{-2} \gamma$ to $\sigma_o$, 
can be given by 
\begin{equation} \label{eq-def-sigma-t}
\sigma (t) = c(t)^{-1}  e^{2(1-t)\varphi}\sigma_o .
\end{equation}
Here $c(t)$ is a normalization function satisfying 
\be \label{eq-ct-bound}
\begin{split}
 c(t) =   \frac{1}{4\pi} \int_{\Sigma} e^{2 (1-t)\varphi} \, d \mu_o 
\geq  \left( \min_\Sigma \,e^{2 \varphi} \right)^{1-t}  .
\end{split}
\ee
The Gauss curvature $K_{\sigma(t)}$ of $\sigma(t)$ satisfies
\begin{equation} 
\begin{split}
c(t)^{-1}  e^{2(1- t) \varphi} K_{ \sigma (t)} = & \ K_{\sigma_0} -  (1-t)  \Delta_{\sigma_o} \varphi \\
= & \ 1-  (1-t)  \Delta_{\sigma_o} \varphi, 
\end{split}
\end{equation}
where $ \Delta_{\sigma_o}$ is the Laplacian on $(\Sigma, \sigma_o)$. 
At $t=0$,  
\begin{equation}
e^{2\varphi}  r_\gamma^{2}  K_\gamma = 1 - \Delta_{\sigma_o}  \varphi.
\end{equation}
It follows that 
\begin{equation}\label{eq-K-conformal}
\begin{split}
K_{\sigma(t)} = & \  c(t) e^{- 2 ( 1 - t) \varphi } \left[ t + ( 1 - t ) e^{2 \varphi} \, r_\gamma^2 \, K_\gamma \right].
\end{split}
\end{equation}

In what follows, suppose $K_\gamma \ge 0$. By \eqref{eq-ct-bound} and \eqref{eq-K-conformal}, 
\begin{equation}\label{eq-K}
\begin{split}
K_{\sigma(t) } \geq \left( \frac{ \min_\Sigma e^{2\phi} }{ \max_\Sigma e^{2\phi} } \right)^{1-t} 
 \left[ t +   ( 1 - t )K_-  \min_\Sigma e^{2 \varphi }  \right].
\end{split}
\end{equation}
Here $ K_- = r_\gamma^2\min_{\Sigma} K_\gamma \ge 0 $.
In particular, \eqref{eq-K} shows 
$ K_{\sigma (t) } > 0 $, $ \forall \ t \in (0, 1] $.

Next, we apply Mantoulidis-Schoen construction \cite{MS} to revise 
$\{ \sigma (t) \}_{t \in [0, 1]}$ into a new path of metrics satisfying property iii) in Section \ref{sec-deformation}.
More precisely, consider a $1$-parameter family of diffeomorphisms $\lbrace \phi_t\rbrace_{t\in [0,1]}$ on $\Sigma$, 
generated by a smooth $t$-dependent vector field $X_t$ which is to be chosen later. 
Let ${\gamma}(t)=\phi_t^*(\sigma(t))$. Then
\begin{equation} \label{eq-dt-gamma-t}
\gamma'(t) = \phi^*_t( \sigma'(t))+\phi^*_t(\mathcal{L}_{X_t} \sigma (t)),
\end{equation}
where $\mathcal{L}$ denotes taking the Lie derivative. 
Hence,
\begin{equation}\label{psi}
\tr_{\gamma(t)}{\gamma}'(t)=\phi^*_t(\tr_{ \sigma (t)} \sigma '(t)+2\dv_{ \sigma (t)} X_t).
\end{equation}
Let  $\psi_t(x)=\psi(t,x)$ to be a smooth function on $[0,1]\times\Sigma$ which is a solution to
\begin{equation}\label{eq-gauge-for-psi}
\Delta_{ \sigma (t)}\psi_t=-\frac{1}{2} \tr_{ \sigma(t)} \sigma'(t),
\end{equation}
for each $t$.
Such a $\psi_t $ exists since  $ \{ \sigma (t) \}_{t \in [0,1]}$ has constant volume which guarantees 
$$\int_\Sigma\frac{1}{2} \tr_{ \sigma (t)} \sigma'(t) \, d \mu_{ \sigma (t)} =0 .  $$
Fix such a $ \psi_t$, let  $X_t=\nabla_{\sigma (t) } \psi_t$ where $\nabla_{\sigma (t) } $ is the gradient with respect to  $\sigma(t)$, 
then 
\be \label{eq-fixed-volume-form}
\tr_{{\gamma (t) }}{\gamma'(t) }=0
\ee 
by \eqref{psi} and \eqref{eq-gauge-for-psi}. 
Note that $ \gamma(0)  = \phi_0^*( \sigma (0) )$ is isometric to $ \sigma (0) = r_\gamma^{-2} \gamma$.
By abusing notation, we denote $  \left( \phi_0^{-1} \right)^* ( \gamma(t) ) $ still by $\gamma (t)$. Then 
$\{ \gamma (t) \}_{t \in [0,1]}$ connects $r_\gamma^{-2} \gamma$ to a round metric, 
has positive Gauss curvature for $ 0 < t \le 1$, and satisfies \eqref{eq-fixed-volume-form}.

Let $ \alpha (t) $ and $ \beta (t)$ be the function associated to $\{ \gamma (t) \}_{t \in [0,1]}$, given in 
\eqref{def-alpha(t)-beta(t)}. 
By \eqref{eq-K}, 
\begin{equation}\label{eq-lower-est-beta-t}
\begin{split}
\beta (t) \geq & \ e^{ - 4 (1-t ) \norm{ \varphi}_0 } 
 \left[ t +   ( 1 - t )K_-  e^{ - 2 \norm{ \varphi }_0 }  \right] . \\ 
\end{split}
\end{equation}
Next we estimate $\alpha (t)$. For the purpose of obtaining the elliptic estimate \eqref{eq-Schauder},
we normalize $ \psi_t $ so that   
\be \label{eq-normalization}
\int_\Sigma \psi_t \, d \mu_{\sigma_o} = 0 ,  \ \forall \, t \in[0,1].
\ee
This can be arranged as $ \psi_t $ is unique up to adding a constant for each $t$.

By the definition of $\alpha (t)$ and   \eqref{eq-dt-gamma-t},
\begin{equation} \label{eq-formula-at}
\begin{split}
\alpha(t) = & \ \frac18 \max_{\Sigma}| \sigma'(t)+\mathcal{L}_{X_t} \sigma (t) |_{\sigma(t)}^2 \\
= & \ \frac18 \max_{\Sigma} \left[ | \sigma'(t) |_{\sigma (t)}^2+2 \langle \sigma'(t),\mathcal{L}_{X_t} \sigma(t)  \rangle_{ \sigma (t) }+|\mathcal{L}_{X_t} \sigma (t) |_{\sigma (t) }^2 \right].
\end{split} 
\end{equation}
Let $ c_t = \ln c(t)$.  By \eqref{eq-def-sigma-t}, 
$  \sigma (t) =   e^{2(1-t)\varphi - c_t }\sigma_o $. 
Hence,  $\sigma'(t) = ( - 2 \varphi - c_t' ) \sigma (t) $ and
\begin{equation}\label{eq-term-1}
| \sigma'(t) |_{\sigma(t)}^2 = 2 ( 2\varphi +c'_t )^2.
\end{equation}
By \eqref{eq-gauge-for-psi} and the fact $ X_t = \nabla_{\sigma (t) } \psi_t $, 
\begin{equation}\label{eq-term-2}
\langle \sigma'(t) ,\mathcal{L}_{X_t} \sigma(t)  \rangle_{ \sigma(t) } =-2\left( 2\varphi + c'_t\right)^2.
\end{equation}
The term $ \mathcal{L}_{X_t} \sigma (t) $ satisfies 
\begin{equation} \label{eq-LX-Hessian}
\mathcal{L}_{X_t} \sigma (t) =2  \Hess_{\sigma (t) } \psi_t.
\end{equation}
Here $ \Hess_{\sigma (t) }$ denotes the Hessian on $(\Sigma, \sigma(t))$.
Since $ \sigma (t) $ is conformal to $ \sigma_o$, the following relation between $ \Hess_{\sigma_o} \psi_t$ and $ \Hess_{\sigma(t)} \psi_t$ 
can be checked directly: 
\be
 \Hess_{ \sigma(t) } \psi_t =  \Hess_{\sigma_o } \psi_t   - (1-t) \left[ d \psi_t   \otimes d  \varphi + d \varphi \otimes d \psi_t 
 - \langle d  \varphi, d \psi_t \rangle_{\sigma_o}   \sigma_o  \right] .
\ee
Therefore,
\begin{equation} \label{eq-Hessian-bd-1}
\begin{split}
 |\Hess_{ \sigma(t)} \psi_t |_{ \sigma_o } \leq & \  |\Hess_{\sigma_o} \psi_t |_{\sigma_o} 
+ (1-t) \left( 2 | d\psi_t  |_{\sigma_o} | d \varphi |_{\sigma_o} +  \sqrt{2} \, | \langle d\varphi,d\psi_t \rangle_{\sigma_o} | \right)  \\
\leq & \  C || \psi_t ||_{C^2(\Sigma) } \left( 1 + (1-t)   \norm{d \varphi}_0 \right) ,
\end{split}
\end{equation}
where $C$ is some constant depending only on $ \sigma_o$.
By \eqref{eq-gauge-for-psi}, \eqref{eq-def-sigma-t} and \eqref{eq-normalization}, $\psi_t$ on
$(\Sigma, \sigma_o)$ satisfies 
\begin{equation}
\Delta_{\sigma_o} \psi_t =  e^{2(1-t)\varphi - c_t } \left(  2 \varphi + c_t' \right)  \ 
\text{and} \ \int_{\Sigma} \psi_t \, d \mu_{\sigma_o} = 0 .
\end{equation}
By the standard elliptic theory, for any fixed $\alpha \in (0,1)$,
\be \label{eq-Schauder}
|| \psi_t ||_{C^{2, \alpha} (\Sigma )} \le C  || \Delta_{\sigma_o} \psi_t  ||_{C^{0, \alpha}  (\Sigma)} ,
\ee
where $C$ only depends on $\sigma_o$ and $\alpha$. 
The H\"{o}lder norm of $\Delta_{\sigma_o} \psi_t $ can be estimated as follows:
\begin{equation}\label{eq-seminorm-laplacian}
\begin{split}
\lsm{ e^{2(1-t)\varphi}} (2\varphi+c'_t)   \rsm_\alpha
&\leq \norm{e^{2(1-t)\varphi}}_{0} \, 2 \lsm\varphi\rsm_\alpha + \lsm{{e^{2(1-t)\varphi}}}\rsm_\alpha \, \norm{2\varphi+c'_t}_{0}.\\
\end{split}
\end{equation}
By the mean value theorem, given any $x, y \in \Sigma$, 
 \begin{equation}
|e^{2(1-t)\varphi(x)}-e^{2(1-t)\varphi(y)}|=2(1-t) |\varphi(x)-\varphi(y)| e^{2(1-t)\xi} ,
\end{equation}
for some $\xi $ lying between $\varphi(x) $ and $\varphi(y)$. Thus,
\begin{equation}
\lsm e^{2(1-t)\varphi}\rsm_\alpha\leq 2(1-t) \norm{ e^{2(1-t) \varphi}}_{0} \lsm \varphi\rsm_\alpha.
\end{equation}
Therefore,
\begin{equation}
\begin{split}
\lsm{e^{2(1-t)\varphi} (2\varphi+c'_t)  }\rsm_\alpha
&\leq 2\lsm\varphi\rsm_\alpha\norm{e^{2(1-t)\varphi}}_{0}\Big(   1+(1-t) \norm{2\varphi+c'_t}_{0} \Big) .
\end{split}
\end{equation}
Consequently,
\begin{equation} \label{eq-c0alpha-norm}
\begin{split}
&  \norm{e^{2(1-t)\varphi}  (2\varphi+c'_t)   }_{C^{0,\alpha} (\Sigma) }=\norm{ e^{2(1-t)\varphi} (2\varphi+c'_t)  }_{0} 
+ \lsm{ e^{2(1-t)\varphi}  (2\varphi+c'_t)  }\rsm_\alpha\\
&\leq  \ \norm{e^{2(1-t)\varphi}}_{0}
\left[ \norm{2\varphi+c'_t}_{0} +  2\lsm\varphi\rsm_\alpha \Big(   1+(1-t) \norm{2\varphi+c'_t}_{0}\Big) \right]  \\
&\leq  \  e^{2(1-t)\norm{\varphi}_{0}}
\left[ 4\norm{\varphi}_{0} + 2 \lsm \varphi\rsm_\alpha\Big(   1+ (1-t)  4  \norm{\varphi}_{0}\Big)  \right], 
\end{split}
\end{equation}
where we also used the fact $ c_t = \ln c(t)$ and 
\begin{equation}
|c_t'|
= \left|  \frac{\int -2\varphi \, e^{2(1-t)\varphi}d\mu_0}{\int  e^{2(1-t)\varphi}d\mu_0} \right|
\leq 2\norm{\varphi}_{0}.
\end{equation}
It follows from \eqref{eq-Hessian-bd-1} -- \eqref{eq-Schauder} and \eqref{eq-c0alpha-norm} that
\begin{equation} \label{eq-hessian-bd-final-1}
\begin{split}
&   |\Hess_{ \sigma(t)} \psi_t |_{ \sigma(t) }=   e^{-2(1-t)\varphi+c_t}|\Hess_{\sigma(t) } \psi_t |_{\sigma_o}\\
&\leq  \ e^{-2(1-t)\varphi+c_t} C_1 || \Delta_{\sigma_o} \psi_t ||_{C^{2, \alpha}  (\Sigma) } \left( 1 + (1-t)   \norm{d \varphi}_0 \right) \\
& \leq  \  C_2 e^{4 (1-t)\norm{\varphi}_{0}}
\Big(  \norm{\varphi}_{0} +  \lsm \varphi\rsm_\alpha + (1-t) \norm{\varphi}_{0}  \lsm \varphi\rsm_\alpha   \Big)
  \left( 1 + (1-t) \norm{d\varphi}_0 \right) .
\end{split}
\end{equation}
Here $ C_i$, $ i = 1, 2,  \ldots $, are constants only depending on $\sigma_o$ and $\alpha$.  
It follows from  \eqref{eq-formula-at} -- \eqref{eq-LX-Hessian} and \eqref{eq-hessian-bd-final-1} that
\be \label{eq-upper-est-alpha-t}
\alpha(t) \le  C_3 e^{8(1-t) \norm{  \varphi}_{0}} 
 \norm{\varphi}_{0, \alpha}^2   \left( 1 + (1-t)  \norm{\varphi}_{0, \alpha}    \right)^2
  \left( 1 + (1-t) \norm{d\varphi}_0 \right)^2, 
\ee
where $ || \varphi ||_{0, \alpha} =  \norm{\varphi}_{0} +  \lsm \varphi\rsm_\alpha  $. 
Equality in \eqref{eq-upper-est-alpha-t} holds if $ \varphi = 0  $  in which case $ \gamma$ is a round metric.

By \eqref{eq-lower-est-beta-t} and \eqref{eq-upper-est-alpha-t}, 
\begin{equation} \label{eq-ab-t-1}
\begin{split}
\sqrt{ \frac{\alpha (t) }{ \beta (t) } } 
\le & \ C_3 \frac{  e^{6  (1-t) \norm{\varphi}_{0}}   \norm{\varphi}_{0, \alpha}   \left( 1 + (1-t)  \norm{\varphi}_{0, \alpha}    \right)
  \left( 1 + (1-t) \norm{d\varphi}_0 \right) }
{ \sqrt{  t +   ( 1 - t )  e^{ - 2 \norm{ \varphi }_0 }  \, K_- }  } .
\end{split}
\end{equation}
Note that
$$
2 \ge \int_0^1 \frac{1}{ \sqrt{ t + (1 - t) e^{ - 2  \norm{ \varphi }_0 } K_- } }  \, d t = \frac{2  }{1 + \sqrt{e^{ - 2  \norm{ \varphi }_0 } K_-} } \ge 1 , 
$$
where we used $ 0 \le K_- \le 1 $.
Therefore, 
\be
\int_0^1 \sqrt{ \frac{\alpha (t) }{ \beta (t) } } \, d t \leq
C_4    e^{6 \norm{\varphi}_{0}}     \norm{\varphi}_{0, \alpha}   \left( 1 +  \norm{\varphi}_{0, \alpha}    \right)
  \left( 1 + \norm{d\varphi}_0 \right)    .
\ee
This proves \eqref{eq-est-zeta-1} by the definition of $\zeta (\gamma)$.
\end{proof}

In the rest of this section, we assume $ K_\gamma > 0 $. Applying results on 
the problem of prescribing Gauss curvature on a sphere from the literature (for instance \cite{Chang_parkcity, CGY93}), one can 
estimate $\zeta(\gamma)$ by the ratio between $   \max_{\Sigma} K_\gamma $ and $ \min_{\Sigma} K_\gamma $.

\begin{prop} \label{prop-zeta-2}
Let $ \gamma $ be a metric with positive Gauss curvature on a $2$-sphere $ \Sigma$. 
Let $ \zeta (\gamma)$ be given in \eqref{eq-def-zeta}. Then
$$
\zeta (\gamma)  \leq C (\kappa)   ,
$$
where $ C (\kappa) $ is a constant depending only on 
$ \displaystyle \kappa =  \frac{ \max_{\Sigma} K_\gamma} { \min_{\Sigma} K_\gamma } \ge 1 $. 
Moreover, there exists a small  $\epsilon > 0 $, such that,  if 
$ \kappa < 1 + \epsilon $, then
$$ \zeta (\gamma) \le  C | \kappa - 1 |, $$
where $C$ is some absolute constant. 
\end{prop}

\begin{proof}
Let $  \tilde \gamma =  r_\gamma^{-2} \gamma $. Then  
$$  \frac{ \max_{\Sigma} K_{ \tilde \gamma } } { \min_{\Sigma} K_{ \tilde \gamma } }  = \kappa . $$
By the Gauss-Bonnet theorem,  $ K_{\tilde \gamma} = 1 $ somewhere on $ \Sigma$. Thus,
\be \label{eq-bound-Gauss}
\kappa^{-1} \le K_{ \tilde \gamma} \le \kappa .
\ee
The function $\varphi$ in \eqref{eq-uniformization-1} satisfies
\be \label{eq-varphi-pde}
\Delta_{\sigma_o} \varphi + K_{\tilde \gamma} e^{2 \varphi } = 1 . 
\ee
Replacing $\gamma $ by $ \Phi^*( \gamma)$ if necessary, where $ \Phi$ is a conformal diffeomorphism on $(\Sigma, \sigma_o)$, one may
assume $\varphi$ satisfies a balancing condition
\be
\int_{\Sigma} x_i e^{2 \varphi} \, d \mu_o = 0 , \ i = 1, 2 , 3 ,
\ee
where $ x_i$ denotes the coordinate function on $ \Sigma$ 
if $ \Sigma $ is identified with the unit sphere $\{ |x| = 1 \}$ in $ \R^3$ (see \cite{Chang_parkcity} for instance).
By \eqref{eq-bound-Gauss} and \cite[Lemma 3.1]{CGY93} (also see (a)' of Chapter 7 in \cite{Chang_parkcity}), 
there exists a constant $ C(\kappa) $, depending on $\kappa$, so that
\be \label{eq-c0-bound-varphi}
|| \varphi ||_{0} \le C (\kappa). 
\ee
It follows from \eqref{eq-bound-Gauss}, \eqref{eq-varphi-pde}, \eqref{eq-c0-bound-varphi} and $L^p$ elliptic estimates that 
$  || \varphi ||_{W^{2,p} } $ is bounded by some constant depending only on $\kappa$ and any chosen $ p > 2$.
By Sobolev embedding theorems, this implies
\be \label{eq-c0-bound-gd-varphi}
|| d \varphi ||_{0} \le C (\kappa, p) ,
\ee
where the constant depends only on $\kappa$ and $p$.
The claim $ \zeta ( \gamma) \le C (\kappa)$ follows from \eqref{eq-est-zeta-1}, \eqref{eq-c0-bound-varphi} and \eqref{eq-c0-bound-gd-varphi}.

Next, suppose $\kappa $ is close to $1$. Then $ || K_{\tilde \gamma } - 1 ||_{0} $ is small 
by \eqref{eq-bound-Gauss}.
In this setting, it was shown on page 433-434 in  \cite{SWW09} that
there exists a constant $\delta > 0 $ such that 
\be \label{eq-SWW}
\begin{split}
|| K_{\tilde \gamma } - 1 ||_{0} \le \delta \Longrightarrow & \ 
 || \varphi ||_{W^{2,2}} \le C || K_{\tilde \gamma } - 1 ||_{0}, \\ 
 \  \ & \  \text{hence} \  || \varphi ||_{0, \alpha} \le C || K_{\tilde \gamma } - 1 ||_{0}  
\end{split}
\ee
for some $ \alpha \in (0,1)$ and $C$ is a constant depending only on $\alpha$. 
Let  $\alpha $ in Proposition \ref{prop-zeta-1}  be given by the $\alpha$ in \eqref{eq-SWW},
the rest of the claim in Proposition \ref{prop-zeta-2} now follows from 
\eqref{eq-est-zeta-1}, \eqref{eq-c0-bound-varphi}, \eqref{eq-c0-bound-gd-varphi} and
\eqref{eq-SWW}.
\end{proof}

Theorem \ref{thm-main} follows from Corollary \ref{cor-est-1} and Proposition \ref{prop-zeta-2}.

\section{Discussion on NNSC fill-ins} \label{sec-NNSC}

Let $ \Sigma$ be a closed $(n-1)$ dimensional manifold, $ n \ge 3$. 
Let $\gamma$ be a metric with positive scalar curvature on $\Sigma$. 
Let $\mathcal{F}(\Sigma,\gamma)$ denote the set of nonnegative scalar curvature (NNSC)
fill-ins of $(\Sigma, \gamma)$, i.e. $\mathcal{F} (\Sigma, \gamma)$ 
consists of $n$ dimensional, compact, connected Riemannian manifolds $(\Omega, g_{_\Omega} )$ with boundary 
such that its boundary $\partial\Omega$, with the induced metric, is isometric to $(\Sigma,\gamma)$, 
and the scalar curvature of $g$ is nonnegative. We are interested in an NNSC fill-in with mean convex boundary. 
Let
$$
\F=\lbrace (\Omega, g_{_\Omega} )\in \FF\,|\, H >0 \rbrace , 
$$
where $H$ is the mean curvature of $\d \Omega$ in $(\Omega, g_{_\Omega})$. 

Following \cite{MM16} (also see \cite{SWWZ, SWW20}), we let
$$\Lambda(\Sigma,\gamma) = 
\sup \left\lbrace  \frac{1}{(n-1) \omega_{n-1} } \int_{\partial\Omega} H  \, d \mu \ \hspace{.1cm}\big| \,  
(\Omega, g_{_\Omega}) \in \F \right\rbrace.$$
Clearly, for any constant $c > 0$, 
\be \label{eq-L-scaling}
\Lambda(\Sigma, c^2 \gamma) = c^{n-2} \Lambda(\Sigma, \gamma) .
\ee

\begin{thm} \label{thm-Lambda}
Let $\gamma$ be a metric with positive scalar curvature on $\Sigma$.
If $ \F  \ne \emptyset $, then
\be \label{eq-lower-L}
  \Lambda(\Sigma, \gamma)  \geq  r_\gamma^{n-1} \left( \frac{ \min_{\Sigma}  R_\gamma  } {(n-1)(n-2) } \right)^\frac12 .
\ee
Here $ r_\gamma$ is the volume radius of $(\Sigma, \gamma)$, i.e. $ | \Sigma |_\gamma = \omega_{n-1} r_\gamma^{n-1}$.
\end{thm}

\begin{proof}
For simplicity, we may assume $r_\gamma = 1$. 
Take $(\Omega, g_{_\Omega}) \in \F$.
Choose  $\gamma (t) = \gamma$, $ 0 \le t \le 1$, and 
use $\{ \gamma (t) \}_{t \in [0,1]} $ and the function $H$, determined by $(\Omega, g_{_\Omega})$,  
in Proposition \ref{prop-mean curv}, we have  
\begin{equation}\label{eq-c-path}
\frac{d\mathcal{H}^2_s}{ds}\geq \frac{n-2}{s} \mathcal{H}^2_s 
+ (n-2) s^{2n-5} \beta_0, \ \ \forall \, s \ge 1 .
\end{equation}
where $ \beta_0 = \frac{1} {(n-1)(n-2) }  \min_{\Sigma}  R_\gamma$.
Fix any $s > 1$, \eqref{eq-c-path} implies 
\be \label{eq-c-mH}
s^{2-n} \mathcal{H}_s^2 \ge  \mathcal{H}_1^2 +  \beta_0 ( s^{n-2} - 1 ). 
\ee

To proceed, note that if  $(\Sigma \times [1, s] , g)$ is attached to $(\Omega, g_{_\Omega} )$ 
by identifying $\Sigma_1 $ with $ \p \Omega$, 
we would get an NNSC fill-in of $(\Sigma_s, s^2 \gamma)$, except the resulting fill-in may not 
be smooth across $\Sigma_1 = \p \Omega$. 
For the moment, suppose this fill-in were smooth. Then, by \eqref{eq-c-mH} and 
the definition of $\Lambda (\Sigma, s^2 \gamma)$, 
\be \label{eq-L-mH}
s^{2-n} \Lambda (\Sigma, s^2 \gamma)^2 \ge  \mathcal{H}_1^2 +  \beta_0 ( s^{n-2} - 1 ). 
\ee
Taking the supremum of the right side of \eqref{eq-L-mH} over $(\Omega, g_{_\Omega} ) \in \F$, 
we obtain
\be \label{eq-L-mH-1}
s^{2-n} \Lambda (\Sigma, s^2 \gamma)^2 \ge \Lambda (\Sigma,  \gamma)^2 +  \beta_0 ( s^{n-2} - 1 ). 
\ee
By \eqref{eq-L-scaling}, the above becomes
\be \label{eq-L-mH-2}
s^{n-2} \Lambda (\Sigma,  \gamma)^2 \ge \Lambda (\Sigma,  \gamma)^2 +  \beta_0 ( s^{n-2} - 1 ), 
\ee
which yields 
$$ \Lambda (\Sigma,  \gamma)^2 \ge \beta_0 , $$ 
giving the estimate in \eqref{eq-lower-L}.
To finish the proof, we note by applying the mollification construction in \cite{M02},
the above mentioned ``singular" fill-ins can be approximated by smooth fill-ins whose normalized  
total mean curvature approaches $\mathcal{H}_s$
(see \cite{MMT18,  SWW20, SWWZ} for instance). This completes the proof. 
\end{proof}

If  $n = 3$, \eqref{eq-lower-L} becomes 
$ 
  \Lambda(\Sigma, \gamma)  \geq  r_\gamma^2 \left(  \min_\Sigma K_\gamma \right)^\frac12 .
$
This can be alternatively derived by isometrically embedding $(\Sigma, \gamma)$ in $ \R^3$, 
making use of the classic Minkowski inequality and the Gauss-Bonnet theorem.

\end{document}